\documentclass[]{article}
\usepackage{amssymb,calc,enumerate,booktabs}
\usepackage{multirow}
\usepackage{etoolbox,rotating,graphicx}
\usepackage{savesym}
\savesymbol{AND}
\usepackage{algorithmic}
\usepackage{algorithm}
\usepackage{caption}
\usepackage{fixltx2e,url}
\usepackage{savesym}
\usepackage{wrapfig}
\usepackage{amsmath}
\usepackage[subrefformat=parens,labelformat=parens]{subcaption}
\usepackage{comment}
\savesymbol{theoremstyle}
\usepackage{amsthm}
\usepackage{arydshln}
\usepackage{cite}
\usepackage{geometry}
\DeclareGraphicsExtensions{.pdf,.jpeg,.png}
\restoresymbol{theorem}{theoremstyle}
\newtheorem{theorem}{Theorem}

\newtheorem{corollary}[theorem]{Corollary}

\newcommand{\Aa}{\mathcal{A}}
\newcommand{\xx}{\mathcal{X}}
\newcommand{\yy}{\mathcal{Y}}
\newcommand{\Gg}{\mathcal{G}}
\newcommand{\ff}{\mathcal{F}}

\begin{document}

\title{Average value of solutions of the bipartite quadratic assignment problem and linkages to domination analysis}

\author{\sc{Ante \'Custi\'c}\thanks{{\tt acustic@sfu.ca}.
Department of Mathematics, Simon Fraser University Surrey,
 250-13450 102nd AV, Surrey, British Columbia, V3T 0A3, Canada}
\and
\sc{Abraham P. Punnen}\thanks{{\tt apunnen@sfu.ca}. Department of Mathematics, Simon Fraser University Surrey,
 250-13450 102nd AV, Surrey, British Columbia, V3T 0A3, Canada}}

\maketitle

\begin{abstract}
In this paper we study the complexity and domination analysis in the context of the \emph{bipartite quadratic assignment problem}. Two variants of the problem, denoted by BQAP1 and BQAP2, are investigated. A formula for calculating the average objective function value $\Aa$ of all solutions is presented whereas computing the median objective function value is shown to be NP-hard. We show that any heuristic algorithm that produces a solution with objective function value at most $\Aa$ has the domination ratio at least $\frac{1}{mn}$. Analogous results for the standard \emph{quadratic assignment problem} is an open question. 
We show that computing a solution whose objective function value is no worse than that of $n^mm^n-{\lceil\frac{n}{\alpha}\rceil}^{\lceil\frac{m}{\alpha}\rceil}{\lceil\frac{m}{\alpha}\rceil}^{\lceil\frac{n}{\alpha}\rceil}$ solutions of BQAP1 or $m^mn^n-{\lceil\frac{m}{\alpha}\rceil}^{\lceil\frac{m}{\alpha}\rceil}{\lceil\frac{n}{\alpha}\rceil}^{\lceil\frac{n}{\alpha}\rceil}$ solutions of BQAP2, is NP-hard for any fixed natural numbers $a$ and $b$ such that $\alpha=\frac{a}{b}>1$. However, a solution with the domination number  $\Omega(m^{n-1}n^{m-1}+m^{n+1}n+mn^{m+1})$ for BQAP1 and $\Omega(m^{m-1}n^{n-1}+m^2n^{n}+m^mn^2)$ for BQAP2, can be found in $O(m^3n^3)$ time. \medskip

\noindent\emph{Keywords:} Quadratic assignment, bilinear programs, domination analysis, heuristics.
\end{abstract}

\section{Introduction}

For a given $m\times n\times m\times n$ array $Q=(q_{ijk\ell})$ and $m\times n$ matrices $c=(c_{ij})$ and $d=(d_{ij})$, the {\it bipartite quadratic assignment problem of type 1} (BQAP1) is to
\begin{align}\nonumber
	\text{Minimize} \qquad &\sum_{i=1}^m\sum_{j=1}^n\sum_{k=1}^m\sum_{\ell=1}^n q_{ijk\ell}x_{ij}y_{k\ell} + \sum_{i=1}^m\sum_{j=1}^n c_{ij}x_{ij} + \sum_{i=1}^m\sum_{j=1}^n d_{ij}y_{ij} \\
	\text{subject to}\quad \  \ & \sum_{j=1}^n x_{ij}=1 \qquad \qquad i=1,2,\ldots,m,\label{x1}\\
	&\sum_{i=1}^m y_{ij}=1 \qquad \qquad j=1,2,\ldots,n, \label{y1}\\
	&x_{ij},\ y_{ij}\in \{0,1\} \qquad  i=1,\ldots,m,\ \ j=1,\ldots,n. \nonumber
\end{align}
Similarly, for a given $m\times m\times n\times n$ array $Q=(q_{ijk\ell})$, and $m\times m$ matrix $c=(c_{ij})$ and $n\times n$ matrix $d=(d_{ij})$, the {\it bipartite quadratic assignment problem of type 2} (BQAP2) is to
\begin{align}
	\text{Minimize} \qquad &\sum_{i=1}^m\sum_{j=1}^m\sum_{k=1}^n\sum_{\ell=1}^n q_{ijk\ell}x_{ij}y_{k\ell} + \sum_{i=1}^m\sum_{j=1}^m c_{ij}x_{ij} + \sum_{i=1}^n\sum_{j=1}^n d_{ij}y_{ij}  \nonumber\\
	\text{subject to}\quad \  \ & \sum_{j=1}^m x_{ij}=1 \qquad \qquad i=1,2,\ldots,m, \label{x2}\\
	&\sum_{i=1}^n y_{ij}=1 \qquad \qquad j=1,2,\ldots,n, \label{y2}\\
	&x_{ij},\ y_{k\ell}\in \{0,1\} \qquad  i,j=1,\ldots,m,\ \ k,\ell=1,\ldots,n. \nonumber
\end{align}
When $m=n$, the problems BQAP1 and BQAP2 are the same. Furthermore, if we impose the additional restriction that $x_{ij}=y_{ij}$ for all $i,j$, both BQAP1 and BQAP2 becomes equivalent to the well-known \emph{quadratic assignment problem} (QAP). Note that the constraints $x_{ij}=y_{ij}$ can be enforced simply by modifying the entries of $Q,c$ and $d$ without explicitly stating the constraints.

The problems BQAP1 and BQAP2 were studied by Punnen and Wang in \cite{PW15} where they proposed efficient heuristic algorithms to solve these problems. They also reported extensive experimental results establishing the quality of their heuristic solutions. If the constraints $x_{ij}, y_{ij}\in \{0,1\} $ are replaced by $0 \leq x_{ij}, y_{ij} \leq 1$ for all $i,j$, in BQAP1 and BQAP2,  we get their corresponding bilinear programming (BLP)~\cite{alt,kono,kky} relaxations, denoted by BLP1 and BLP2, respectively. It is well known that there exists an optimal solution to BLP which is an extreme point of the underlying convex polytope~\cite{alt,kono,kky}. In the case of BLP1 and BLP2, the coefficient matrix of the constraints is totally unimodular and hence all extreme points are of 0-1 type. Thus, BLP1 and BLP2 are respectively equivalent to BQAP1 and BQAP2. Therefore, BQAP1 and BQAP2 can also be solved using any general purpose algorithms for BLP.

BQAP1 and BQAP2 are known to be strongly NP-hard~\cite{PW15}. To the best of our knowledge, theoretical properties of these problems are not investigated thoroughly in the literature. In this paper, we study the complexity of BQAP1 and BQAP2 from the point of view of domination analysis~\cite{a1,gp}. 
Many researchers considered such analysis for various combinatorial optimization problems~\cite{a1,an1,gp,gr1,gg1r,gutin3,gutin2,gg2r,h1,k1,kh1,p3,p2,PSK,rb1,sn1,sd2,sn3,t1,v1,z1}. Domination analysis is also linked to exponential neighborhoods~\cite{x1} and very large-scale neighborhood search~\cite{a14,MP09}.

In this paper, we provide a closed form formula to calculate the average value of all solutions of BQAP1 and BQAP2 and show that there are at least $n^{m-1}m^{n-1}$ and $m^{m-1}n^{n-1}$ solutions respectively for BQAP1 and BQAP2 that have objective function value equal to or worse than the average value of all solutions. For the standard quadratic assignment problem, although a closed form formula exists to calculate the average value of solutions, establishing non-trivial domination results is an open problem~\cite{an1,gutin2,sn3}. We then show that some heuristics that works well in practice could produce solutions with objective function value  worse than the average value of solutions and also provide simple polynomial algorithms that guarantee a solution with objective function value no worse than the average value of solutions. Unlike the average value, computing the median value of solutions for BQAP1 and BQAP2 are shown to be NP-hard. Further, we show that computing a solution whose objective function value is no worse than that of $n^mm^n-{\lceil\frac{n}{\alpha}\rceil}^{\lceil\frac{m}{\alpha}\rceil}{\lceil\frac{m}{\alpha}\rceil}^{\lceil\frac{n}{\alpha}\rceil}$ solutions of BQAP1 is NP-hard for any fixed natural numbers $a$ and $b$ such that $\alpha=\frac{a}{b}>1$. Likewise,  computing a solution whose objective function value is no worse than that of $m^mn^n-{\lceil\frac{m}{\alpha}\rceil}^{\lceil\frac{m}{\alpha}\rceil}{\lceil\frac{n}{\alpha}\rceil}^{\lceil\frac{n}{\alpha}\rceil}$ solutions of BQAP2 is also shown to be NP-hard for any fixed natural numbers $a$ and $b$ such that $\alpha=\frac{a}{b}>1$.

Let $\xx_1$ denotes the set of all 0-1 $m\times n$  matrices satisfying \eqref{x1}, and $\xx_2$ denotes the set of 0-1 $m\times m$ 0-1 matrices satisfying \eqref{x2}. Similarly, let $\yy_1$ be the set of  all $m \times n$ 0-1 matrices satisfying \eqref{y1} and $\yy_2$ be the set of all $n \times n$ 0-1 matrices satisfying \eqref{y2}. Also, $\ff_1$ and $\ff_2$ denote the sets of feasible solutions of BQAP1 and BQAP2, respectively. Note that $|\ff_1|=n^mm^n$ and $|\ff_2|=m^mn^n$. Let $M=\{1,2,\ldots,m\}$ and $N=\{1,2,\ldots,n\}$. For given cost arrays and $x\in\xx_1$, $y\in\yy_1$, let $f_1(x,y)$ denotes the objective function of BQAP1 and for $x\in\xx_2$, $y\in\yy_2$, let $f_2(x,y)$ denotes the objective function of BQAP2.

\section{Average value of solutions and domination properties}

In this section we present extensions of various results proved in \cite{PSK} in the context of the unconstrained bipartite binary quadratic programs to the problems BQAP1 and BQAP2.

Given appropriate size cost arrays $Q,c,d$, let $\Aa_1(Q,c,d)$ and $\Aa_2(Q,c,d)$ be the average objective function value of all feasible solutions of BQAP1 and BQAP2, respectively.

\begin{theorem}\label{th1}
 ${\displaystyle 	\Aa_1(Q,c,d)=\frac{1}{mn}\sum_{i=1}^m\sum_{j=1}^n\sum_{k=1}^m\sum_{\ell=1}^n q_{ijk\ell}+ \frac{1}{n}\sum_{i=1}^m\sum_{j=1}^nc_{ij}+ \frac{1}{m}\sum_{i=1}^m\sum_{j=1}^nd_{ij}}$, and\\
 ${\displaystyle \Aa_2(Q,c,d)=\frac{1}{mn}\sum_{i=1}^m\sum_{j=1}^m\sum_{k=1}^n\sum_{\ell=1}^n q_{ijk\ell}+ \frac{1}{m}\sum_{i=1}^m\sum_{j=1}^mc_{ij}+ \frac{1}{n}\sum_{i=1}^n\sum_{j=1}^nd_{ij}}$.
\end{theorem}
\begin{proof}
Since for every $(x,y)\in\ff_1$ we have that $x_{ij},y_{ij}\in\{0,1\}$ for all $i,j$, it follows that the objective function value $f_1(x,y)$ is a sum of multiple cost elements $c_{ij}$, $d_{ij}$ and $q_{ijkl}$, where each cost element appears at most once. Cost element $c_{ij}$ appears in the objective value sum of $(x,y)$ if and only if $x_{ij}=1$, and there is exactly $n^{m-1}m^n$ such solutions $(x,y)$ in $\ff_1$ for every fixed $i\in M$ and $j\in N$. Similarly, cost element $d_{ij}$ appears in the objective value sum of $(x,y)$ if and only if $y_{ij}=1$, and there is exactly $n^{m}m^{n-1}$ such solutions in $\ff_1$. Lastly, cost element $q_{ijkl}$ appears in the objective value sum of $(x,y)$ if and only if $x_{ij}=1$ and $y_{kl}=1$, and there is exactly $n^{m-1}m^{n-1}$ such solutions $(x,y)$ in $\ff_1$. Hence if we sum objective function values for all feasible solutions in $\ff_1$ we get
\[
\sum_{(x,y)\in\ff_1}f_1(x,y)=n^{m-1}m^{n-1}\sum_{\substack{(i,j,k,\ell)\in \\ M\times N\times M \times N}} q_{ijk\ell}+ n^{m-1}m^{n}\sum_{\substack{(i,j)\in\\ M\times N}}c_{ij}+ n^{m}m^{n-1}\sum_{\substack{(i,j)\in\\ M\times N}}d_{ij}.
\]
Since $|\ff_1|=n^mm^n$ the average objective function value is equal to
\[
\Aa_1(Q,c,d)=\frac{\sum_{(x,y)\in\ff_1}f_1(x,y)}{|\ff_1|}=\frac{1}{mn}\sum_{\substack{(i,j,k,\ell)\in \\ M\times N\times M \times N}} q_{ijk\ell}+ \frac{1}{n}\sum_{\substack{(i,j)\in\\ M\times N}}c_{ij}+ \frac{1}{m}\sum_{\substack{(i,j)\in\\ M\times N}}d_{ij}.
\]
The expression for $\Aa_2(Q,c,d)$ can be obtained in a similar way.
\end{proof}

The formula  for $\Aa_1(Q,c,d)$ and $\Aa_2(Q,c,d)$ discussed in Theorem~\ref{th1} can also be deduced using a probabilistic argument. Consider $x_{ij}$ and $y_{ij}$ as a 0-1 random variables with probability of $x_{ij}=1$ as $1/m$ and probability of $y_{ij}=1$ as $1/n$. Then the formula for $\Aa_1(Q,c,d)$ and $\Aa_2(Q,c,d)$ follows from the linearity property of expectations.

We now provide a lower bound on the size of the set of feasible solutions that are no better than the average. This result is interesting since we do not assume any specific probability distribution on the objective function values of the solutions. Moreover, such results for the standard quadratic assignment problem is an open question~\cite{an1,gutin2,sn3}.  Let $\Gg_i=\{(x,y)\colon x\in\xx_i,y\in\yy_i, f_i(x,y)\geq A_i(Q,c,d)\}$, $i=1,2$.

\begin{theorem}\label{thm:dom}
	$|G_1|\geq n^{m-1}m^{n-1}$ and $|G_2|\geq m^{m-1}n^{n-1}.$
\end{theorem}
\begin{proof}
	We will prove the result for BQAP1. Consider the equivalence relation $\sim$ on $\ff_1$, where $(x,y)\sim(x',y')$ if and only if there exist $a,b\in\{0,1,\ldots,n-1\}$ such that $x_{ij}=x'_{i(j+a \mod n)}$ and $y_{ij}=y'_{(i+b \mod m)j}$ for all $i,j$. Note that $\sim$ partitions $\ff_1$ into equivalence classes of size $mn$. Hence there is $n^{m-1}m^{n-1}$ such classes. Let $S$ be any such equivalence class. Note that for every $i\in M,j\in N$, there are exactly $m$ solutions in $S$ for which $x_{ij}=1$, and there are exactly $n$ solutions in $S$ for which $y_{ij}=1$. Furthermore, for every $i,k\in M,j,\ell\in N$, there is exactly one solution in $S$ for which $x_{ij}y_{k\ell}=1$. Hence
\[
	\sum_{(x,y)\in S}f_1(x,y)=\sum_{\substack{(i,j,k,\ell)\in \\ M\times N\times M \times N}} q_{ijk\ell}+m\sum_{\substack{(i,j)\in\\ M\times N}}c_{ij}+ n\sum_{\substack{(i,j)\in\\ M\times N}}d_{ij}=mn\Aa (Q,c,d).
\]
Since $|S|=mn$, there exist at least one $(x,y)\in S$ such that $f_1(x,y)\geq \Aa(Q,c,d)$, which proves the theorem for BQAP1. The proof for BQAP2 follows in a similar way and is omitted.
\end{proof}

To show that the bound presented in Theorem~\ref{thm:dom} is tight, consider the following instance. Let arrays $Q,c,d$ be such that all of their elements are 0, except $q_{i'j'k'\ell'}=1$ for some fixed $i',j',k',\ell'$. Since exactly one element from every equivalence class defined by $\sim$ is not better than average, the tightness follows.

The proof technique of Theorem~\ref{thm:dom} can also be used to obtain a feasible solution with the objective function value less than or equal to the average. In the proof of Theorem~\ref{thm:dom} we show that in every equivalence class defined by $\sim$ there is a feasible solution with the objective function value greater than or equal to the average $\Aa_1(Q,c,d)$. By the same reasoning we know that in every such class there is a feasible solution with objective function value less than or equal to the average. For example, given $a\in N\ (M)$, $b\in M\ (N)$ let $(x^a,y^b)\in\ff_1\ (\ff_2)$ be defined  as 
\[
x_{ij}^a=
\begin{cases}
	1 & \text{if } j=a, \\
	0 & \text{otherwise},
\end{cases}\quad \text{and}\quad
y_{ij}^b=
\begin{cases}
	1 & \text{if } i=b, \\
	0 & \text{otherwise}.
\end{cases}
\]
Then $(x^{a_1},y^{b_1})\sim(x^{a_2},y^{b_2})$ for every $a_1,a_2\in N\ (M)$ and $b_1,b_2\in M\ (N)$, and  $f_1(x^a,y^b)=\sum_{i\in M,j\in N} q_{iabj}+\sum_{i\in M}c_{ia}+\sum_{j\in N}d_{bj}\ (=f_2(x^a,y^b))$. 

\begin{corollary}\label{cor:ab_av}
	For every $(Q,c,d)\in \ff_1$ and $(Q',c',d')\in\ff_2$ we have that
	\[\min_{a\in N,b\in M}\{f_1(x^a,y^b)\}\leq\Aa_1(Q,c,d)\leq\max_{a\in N,b\in M}\{f_1(x^a,y^b)\},\] 
\[\min_{a\in M,b\in N}\{f_2(x^a,y^b)\}\leq\Aa_2(Q',c',d')\leq\max_{a\in M,b\in N}\{f_2(x^a,y^b)\}.\]
\end{corollary}
Corollary~\ref{cor:ab_av} enables us to find a feasible solution for BQAP1 (BQAP2) with objective function value no worse than $\Aa_1(Q,c,d)$ ($\Aa_2(Q,c,d)$) in $O(m^2n^2)$ time. Note that any equivalence class defined by $\sim$ can be used to obtain inequalities as in the corollary above.

Let $(x^\Gamma,y^\Gamma)\in\ff$ be a solution produced by a heuristic $\Gamma$, where $\ff$ denotes the set of feasible solutions. Let $\Gg^\Gamma=\{(x,y)\in \ff\colon f(x,y)\geq f(x^\Gamma,y^\Gamma)\}$, and $\mathcal{I}$ be the collection of all instances of the problem with instance size parameters $m$ and $n$. Then
\[
	\inf_{I\in\mathcal{I}}\left|\Gg^\Gamma\right| \text{ \ and \ }  \inf_{I\in\mathcal{I}}\frac{|\Gg^\Gamma|}{|\ff|},
\]
are called \emph{domination number} and \emph{domination ratio} of $\Gamma$, respectively~\cite{a1,gp}.

\begin{theorem}
	For any fixed natural numbers $a$ and $b$ such that $\alpha=\frac{a}{b}>1$ no polynomial time algorithm for BQAP1 \textup{(}BQAP2\textup{)} can have domination number more than $n^mm^n-{\lceil\frac{n}{\alpha}\rceil}^{\lceil\frac{m}{\alpha}\rceil}{\lceil\frac{m}{\alpha}\rceil}^{\lceil\frac{n}{\alpha}\rceil}$ \textup{(}$m^mn^n-{\lceil\frac{m}{\alpha}\rceil}^{\lceil\frac{m}{\alpha}\rceil}{\lceil\frac{n}{\alpha}\rceil}^{\lceil\frac{n}{\alpha}\rceil}$\textup{)}, unless P=NP.
\end{theorem}
\begin{proof}
	We prove the statement for BQAP1, and in the case of BQAP2, the proof follows analogously.

	Let natural numbers $a$ and $b$ be such that $\alpha=\frac{a}{b}>1$. We show that a polynomial algorithm $\Omega$ for  BQAP1 with domination number at least $n^mm^n-{\lceil\frac{n}{\alpha}\rceil}^{\lceil\frac{m}{\alpha}\rceil}{\lceil\frac{m}{\alpha}\rceil}^{\lceil\frac{n}{\alpha}\rceil}+1$ can be used to compute an optimal solution of  BQAP1. Consider an arbitrary instance $(Q,c,d)$ of BQAP1, and let $Q^*=(q^*_{ij})$ be an $abm\times abn\times abm\times abn$ array such that
\[
q^*_{ijk\ell}=\begin{cases}
                        q_{ijk\ell} & \text{if } i,k\in M \text{ and } j,\ell\in N,\\
                        0 & \text{otherwise.}
                    \end{cases}
\]
Furthermore, let $L$ be a large number and let $c^*=(c^*_{ij})$ and $d^*=(d^*_{ij})$ be $abm\times abn$ matrices such that
\[
c^*_{ij}=\begin{cases}
                        c_{ij} & \text{if } i\in M \text{ and } j\in N,\\
				0 & \text{if } i\notin M \text{ and } j\notin N,\\
                        L & \text{otherwise}
                    \end{cases} \quad \text{  and  } \quad
d^*_{ij}=\begin{cases}
                        d_{ij} & \text{if } i\in M \text{ and } j\in N,\\
                        0 & \text{if } i\notin M \text{ and } j\notin N,\\
                        L & \text{otherwise.}
                    \end{cases}
\]
The BQAP1 instances $(Q^*,c^*,d^*)$ and $(Q,c,d)$ are equivalent. In particular, from any optimal solution for $(Q^*,c^*,d^*)$ an optimal solution for $(Q,c,d)$ can be recovered. Hence, the number of optimal solutions of $(Q^*,c^*,d^*)$ is at least $(abn-n)^{abm-m}(abm-m)^{abn-n}$. Therefore, the number of non-optimal solutions is at most $(abn)^{abm}(abm)^{abn}-(abn-n)^{abm-m}(abm-m)^{abn-n}$. Let $(x',y')$ denotes the output of $\Omega$, and assume that $(x',y')$ is no worse than $(abn)^{abm}(abm)^{abn}-{\lceil\frac{abn}{\alpha}\rceil}^{\lceil\frac{abm}{\alpha}\rceil}{\lceil\frac{abm}{\alpha}\rceil}^{\lceil\frac{abn}{\alpha}\rceil}+1=(abn)^{abm}(abm)^{abn}-(b^2n)^{b^2m}(b^2m)^{b^2n}+1$ solutions. From $a>b$, it follows that $(abn-n)^{abm-m}(abm-m)^{abn-n}$ is greater than $(b^2n)^{b^2m}(b^2m)^{b^2n}$, therefore $(x',y')$ is an optimal solution. Using $(x',y')$ we can find an optimal solution for $(Q,c,d)$, which is in contradiction to the fact that BQAP1 is NP-hard.
\end{proof}

\begin{theorem}
	Computing a median of the objective function values of both BQAP1 and BQAP2 is NP-hard.
\end{theorem}
\begin{proof}
Consider the NP-complete PARTITION problem: Given $n$ positive integers $a_1,a_2,\ldots,a_n$, determine if there exists a partition of $N=\{1,2,\ldots,n\}$ into $S_1$ and $S_2$ such that $\sum_{i\in S_1}a_i=\sum_{i\in S_2}a_i$. From an instance of PARTITION we will build an instance $I$ of both BQAP1 and BQAP2 (i.e.\@ $n=m$), such that the median of objective values is $\sum_{i\in N}a_i/2$ if and only if starting PARTITION instance $I$ is a 'yes' instance.

Without loss of generality we can assume than $n$ is even. Let $n\times n\times n\times n$ array $Q$ and $n\times n$ matrices $c$ and $d$ constitute an instance of BQAP1 (BQAP2), where $Q$ and $d$ are null arrays, and $c$ is given by
\[
c_{ij}=\begin{cases}
	a_i & \text{if } j\leq \frac{n}{2},\\
	0 & \text{otherwise.}
\end{cases}
\]
Note that we can partition the set of feasible solutions $\ff$ into pairs of solutions $(x,y)$, $(x',y')$ by the rule $x_{ij}=x'_{i(j+\frac{n}{2}\mod n)}$ and $y_{ij}=y'_{ij}$, for all $i,j\in N$. Note that for every such pair of solutions either $f(x,y)\leq \sum_{i\in N}a_i/2\leq f(x',y')$ or $f(x',y')\leq \sum_{i\in N}a_i/2\leq f(x,y)$. Hence if there is a solution $(x,y)$ for which $f(x,y)=\sum_{i\in N}a_i/2$, then $\sum_{i\in N}a_i/2$ is a median objective value. Every value $\alpha$ appears as an objective value of some solutions for $(Q,c,d)$ if and only if there exists $S\subseteq N$ such that $\sum_{i\in S}a_i=\alpha$. Hence the theorem follows from NP-completeness of the PARTITION problem.
\end{proof}

\section{Domination analysis of heuristics}

Let us now analyze some heuristics for BQAP1 and BQAP2 from the domination analysis~\cite{a1,gp,h1,t1,z1} point of view. First, we consider two local search heuristics that performed very well in experimental analysis~\cite{PW15}.

\subsection{Swap based neighborhoods}

Let $(x,y)$ be a feasible solution of the BQAP1 (BQAP2). Then a \emph{swap on $x$} operation, denoted by $swapx(i,j)$ and applied on $(x,y)$, produces a feasible solution $(x',y)$ such that
\[
x'_{st}=\begin{cases}
	x_{st} & \text{for } s\neq i, \\
	1 & \text{for } s=i,\ t=j,\\
	0 & \text{otherwise.}
\end{cases}
\]
Analogously, swap on $y$ ($swapy(i,j)$) operation on $(x,y)$ produces a feasible solution $(x,y')$ such that
\[
y'_{st}=\begin{cases}
	y_{st} & \text{for } t\neq j, \\
	1 & \text{for } s=i,\ t=j,\\
	0 & \text{otherwise.}
\end{cases}
\]
Then we define the \emph{swap} neighborhood to be the collection  of all solutions obtained by applying one $swapx(i,j)$ or one $swapy(i,j)$. Cardinality of the swap neighborhood for BQAP1 is $2mn-m-n+1$ and for BQAP2 is $m^2+n^2-m-n+1$.

Note that two swaps on $x$, say $swapx(i,j)$ and $swapx(k,l)$ for $i\neq k$, are independent in the sense that they can be applied concurrently or one after another yielding the same solution. Hence, applying a set of swaps $swapx(i_1, j_1),\ldots,swapx(i_k,j_k)$, where $i_s\neq i_t$ for $s\neq t$, we call a \emph{concurrent swap on $x$}. The set of all concurrent swaps on $x$ we denote by $cswap(x)$. Analogously, concurrent swaps on $y$, denoted by $cswap(y)$, are defined. Then we define the \emph{concurrent swap} neighborhood to be the collection of solutions obtained by applying one element of  $cswap(x)\cup cswap(y)$.
Note that the size of the concurrent swap neighborhood is $n^m+m^n-1$ for BQAP1 and $m^m+n^n-1$ for  BQAP2. Nevertheless, it can be searched in polynomial time, since the best concurrent swap on $x$ is made of swaps $swapx(1,j_1),\ldots,swapx(m,j_m)$, where $swapx(k,j_k)$ is the best improving swap among $swapx(k,1),\ldots,swapx(k,n)$ ($swapx(k,m)$ in the case of BQAP2). The best concurrent swap on $y$ can be found in the same manner. Note that for a fixed $x$ we can find a solution $(x,y')$ with the smallest objective value in polynomial time, by searching the concurrent swaps on $y$ for some feasible solution $(x,y)$. Analogously, for a fixed $y$, we can find in polynomial time the smallest objective value solution $(x',y)$.

Lastly we will define the \emph{optimized swap} neighborhood. The \emph{optimized swap on $x$} $oswapx(i,j)$ applied on a feasible solution $(x,y)$ consists of applying $swapx(i,j)$ on $(x,y)$ to obtain $(x',y)$, and then replacing $y$ with $y'$ so that the objective value of $(x',y')$ is minimized. Optimized swap on $y$ ($oswapy(i,j)$) is defined analogously. Then the optimized swap neighborhood is defined to be the collection of solutions obtained by applying one $oswapx(i,j)$ or one $oswapy(i,j)$. As noted above, given $x$ we can find in polynomial time such $y$ that minimizes the objective value of $(x,y)$. Hence the optimized swap neighborhood can be searched in polynomial time. Furthermore, note that the size of the neighborhood, and hence also the lower bound for the domination number of the local optimum based on the optimized swap neighborhood, is $\Theta(nm^{n+1}+mn^{m+1})$ and $\Theta(m^2n^n+n^2m^m)$ for BQAP1 and BQAP2, respectively.

Experimental analysis of the local search approaches for BQAP1 based on the three neighborhoods described above is presented in \cite{PW15}. Our goal here is to analyze the local search optimal solution quality in terms of domination and compared to the average solution value. Local search based on the concurrent neighborhood starting from a feasible solution $(x^0,y^0)$ operates as follows. We fix $x^0$ (or maybe $y^0$) and find $y^1$ that minimizes $(x^0,y^1)$. Then $y^1$ is fixed and we look for $x^1$ that minimizes $(x^1,y^1)$. Then again $x^1$ is fixed, and so on. This process is repeated until a local optimum is obtained. Such heuristic approach falls into the category of \emph{alternating algorithms}. Quality of a solution of such alternating algorithm for the \emph{bipartite boolean quadratic programming problem} (BBQP) compared to the average solution value was investigated in \cite{PSK}.

\begin{theorem}
	The objective function value of a locally optimal solution for BQAP1 and BQAP2 based on the swap neighborhood and the concurrent swap neighborhood could be arbitrary bad and could be worse than $\Aa_1(Q,d,c)$ and $\Aa_2(Q,d,c)$, respectively.
\end{theorem}
\begin{proof}
	Let $L$ be a large  and $\epsilon >0$ a small number, and consider cost arrays $Q,c,d$ of the BQAP1 (BQAP2) such that all of its elements are equal to $0$, except that $c_{12}$ and $d_{21}$ are equal $-\epsilon$, and $q_{1111}=-L$. Let a feasible solution $(x^0,y^0)$ be such that $x^0_{12}=y^0_{21}=1$. Then $(x^0,y^0)$ is a local optimum of both swap and concurrent swap neighborhoods. We have that the objective function value of $(x^0,y^0)$ is $-2\epsilon$, while the optimal solution value is $-L$.
\end{proof}

Although we cannot argue that a swap based local search will give us a solution no worse than the average, we can use it to improve the domination number of an already good solution. 

\begin{theorem}
	A feasible solution $(x,y)$  with the domination number  $\Omega(m^{n-1}n^{m-1}+m^{n+1}n+mn^{m+1})$ for BQAP1 and $\Omega(m^{m-1}n^{n-1}+m^2n^{n}+m^mn^2)$ for BQAP2, can be found in $O(m^3n^3)$ time.
\end{theorem}
\begin{proof}
We show that a solution with the domination number described in the statement of the theorem can be obtained in the desired running time by searching the optimized swap neighborhood of a solution that is no worse than the average. 

Let $(x^*,y^*)\in \ff_1$ be a solution of BQAP1 such that $f_1(x^*,y^*)\leq \Aa_1(Q,c,d)$.
Such solution can be found in $O(m^2n^2)$ time using Corollary~\ref{cor:ab_av} (or Theorem~\ref{thm:ro}). From the proof of Theorem~\ref{thm:dom} we know that there exists a set $S_\sim$ of $m^{n-1}n^{m-1}$ solutions, with one from every class defined by the equivalence relation $\sim$, such that $f_1(x,y)\geq\Aa_1(Q,c,d)\geq f_1(x^*,y^*)$ for every $(x,y)\in S_\sim$. Let $S_x$ denotes the set of solutions $(x',y)$ where $x'$ can be obtained from $x^*$ by $swapx(i,j)$ for some $i\in M,j\in N$ and $y$ is an arbitrary element of $\yy_1$. Let $S_y$ be defined analogously. Note that $S_x\cup S_y$ is the optimized neighborhood of $(x^*,y^*)$. $S_x\cup S_y$ can be searched in $O(m^3n^3)$ time, and the result of the search has the objective function value less or equal than  every $(x,y)\in S_\sim \cup S_x \cup S_y$. By elementary enumerations it can be calculated that $|S_x|=(m(n-1)+1)m^n$, $|S_y|=(n(m-1)+1)n^m$, $|S_\sim \cap S_x|\leq (m(n-1)+1)m^{n-1}$, $|S_\sim \cap S_y|\leq (n(m-1)+1)n^{m-1}$ and $|S_x \cap S_y|=2mn-m-n+1$. Now we can calculate
\begin{align*}
|S_\sim \cup S_x\cup S_y|\geq\  &|S_\sim|+|S_x|+|S_y|-|S_\sim \cap S_x|-|S_\sim \cap S_y|-|S_x \cap S_y|\\
\geq\ &m^{n-1}n^{m-1}+(m-1)(m(n-1)+1)m^{n-1}+(n-1)(n(m-1)+1)n^{m-1}\\
&-2mn+m+n-1,
\end{align*}
which proves the theorem for BQAP1. The statement for BQAP2 can be proved in the same way.
\end{proof}


\subsection{Rounding procedure}

In this section we describe a simple rounding procedure that can always be used to obtain a solution which is no worse than average. Similar rounding procedure for the unconstrained bipartite boolean quadratic program was investigated from both theoretical and experimental point of view in \cite{PSK}.

Let $(x,y)$ be a feasible solution of the BLP1, the bilinear programming relaxation of BQAP1. By the extreme point optimality property of bilinear programming problems, we can find a feasible solution $(x^*,y^*)$ of BQAP1 with objective function value no worse than that of $(x,y)$. Exploiting the special properties of BQAP1, we present below a simple rounding scheme to obtain such $(x^*,y^*)$ from $(x,y)$. Define

\[
x^*_{ij}=
\begin{cases}
	1 & \text{for one arbitrary } j \text{ that minimizes } c_{ij}+\sum_{k\in M, \ell\in N}q_{ijk\ell}y_{k\ell}, \\
	0 & \text{otherwise},
\end{cases}\]
and then 
\[
y^*_{ij}=
\begin{cases}
	1 & \text{for one arbitrary } i \text{ that minimizes } d_{ij}+\sum_{k\in M, \ell\in N}q_{k\ell ij}x^*_{ij}, \\
	0 & \text{otherwise}.
\end{cases}
\]
Note that $y^*$ is an optimal 0-1 matrix  when $x$ is fixed at $x^*$, and $x^*$ is obtained by rounding $x$. Hence, the above rounding procedure we call \emph{round-$x$ optimize-$y$} (RxOy). Similarly, \emph{round-$y$ optimize-$x$} (RyOx) procedure can also be defined. We omit the details. Furthermore, both RxOy and RyOx schemes can be obtained for BQAP2 with appropriate modifications.

\begin{theorem}\label{thm:round}
	If a feasible solution $(x^*,y^*)$ of BQAP1 \textup{(}BQAP2\textup{)} is obtained by RxOy or RyOx procedure from a feasible solution $(x,y)$ of BLP1 \textup{(}BLP2\textup{)}, then $f_1(x^*,y^*)\leq f_1(x,y)$ \textup{(}$f_2(x^*,y^*)\leq f_2(x,y)$\textup{)}.
\end{theorem}
\begin{proof}
		We prove the statement for BQAP1, and in the case of BQAP2 the proof follows analogously.
	\begin{align*}
		f_1(x,y)&=\sum_{i\in M}\sum_{j\in N}\sum_{k\in M}\sum_{\ell\in N} q_{ijk\ell}x_{ij}y_{k\ell} + \sum_{i\in M}\sum_{j\in N} c_{ij}x_{ij} + \sum_{i\in M}\sum_{j\in N} d_{ij}y_{ij}\\
		&=\sum_{i\in M}\left[ \sum_{j\in N}\left( \sum_{k\in M}\sum_{\ell\in N} q_{ijk\ell}y_{k\ell} + c_{ij}\right) x_{ij}\right] + \sum_{i\in M}\sum_{j\in N} d_{ij}y_{ij} \\
	&\geq \sum_{i\in M}\sum_{j\in N}\left( \sum_{k\in M}\sum_{\ell\in N} q_{ijk\ell}y_{k\ell} + c_{ij}\right) x^*_{ij} + \sum_{i\in M}\sum_{j\in N} d_{ij}y_{ij} \\
	&= \sum_{i\in M}\sum_{j\in N} c_{ij}x^*_{ij} + \sum_{\ell\in N}\left[\sum_{k\in M}\left( \sum_{i\in M}\sum_{j\in N} q_{ijk\ell}x^*_{ij} + d_{k\ell}\right) y_{k\ell}\right]\\
	&\geq \sum_{i\in M}\sum_{j\in N} c_{ij}x^*_{ij} + \sum_{k\in M}\sum_{\ell\in N}\left( \sum_{i\in M}\sum_{j\in N} q_{ijk\ell}x^*_{ij} + d_{k\ell}\right) y^*_{k\ell}\\
	&=f_1(x^*,y^*)
	\end{align*}
\end{proof}

\begin{theorem}\label{thm:ro}
	A feasible solution $(x',y')$ of BQAP1 \textup{(}BQAP2\textup{)}  that satisfies $f_1(x',y')\leq \Aa_1(Q,d,c)$ \textup{(}$f_2(x',y')\leq \Aa_2(Q,d,c)$\textup{)} can be obtained in $O(m^2n^2)$ time.
\end{theorem}

\begin{proof}
	Let $x_{ij}=1/n$, $y_{ij}=1/m$ $i\in M, j\in N$ be a feasible solution of BLP1. It is easy to verify that $f_1(x,y)=\Aa_1(Q,c,d)$. Let $(x',y')$ be a solution of BQAP1 obtained from $(x,y)$ by RxOy or RyOx. Then the theorem follows from Theorem~\ref{thm:round} and the fact that RxOy and RyOx take $O(m^2n^2)$ time. (A proof for BQAP2 works in the same way.)
\end{proof}

\section*{Acknowledgment}

This research work was supported by an NSERC discovery grant and an NSERC discovery accelerator supplement awarded to Abraham P. Punnen.

\end{document}